\documentclass[12pt]{amsart}
\usepackage[all,cmtip]{xy}
\usepackage{amscd,amsmath,latexsym, hyperref, color}
\usepackage{mathtools}
\usepackage{graphicx}
\usepackage{enumitem}
\setenumerate{label=\normalfont(\alph*)}

\newtheorem{thm}{Theorem}[section]
\newtheorem{cor}[thm]{Corollary}
\newtheorem{lem}[thm]{Lemma}
\newtheorem{prop}[thm]{Proposition}

\theoremstyle{definition}

\theoremstyle{remark}
\newtheorem{rem}[thm]{Remark}
\newtheorem*{ack}{Acknowledgments}

\def\z2{{\bf Z}$_{2}$}
\def\zl2{{\bf Z}$_{(2)}$}

\begin{document}

\title[Higher torsion in the homotopy groups of Moore spaces]{Infinite families of higher torsion in the homotopy groups of Moore spaces}

\author[S. Amelotte]{Steven Amelotte}
\address{Institute for Computational and Experimental Research in Mathematics, Brown University, 121 South Main Street, Providence, RI 02903, USA}
\email{steven\_amelotte@brown.edu} 

\author[F. R. Cohen]{Frederick R. Cohen}
\address{Department of Mathematics, University of Rochester, Rochester, NY 14625, USA}
\email{fred.cohen@rochester.edu} 

\author[Y. Luo]{Yuxin Luo}
\address{Department of Mathematics, University of Rochester, Rochester, NY 14625, USA}
\email{yuxinluo@rochester.edu}


\begin{abstract}
We give a refinement of the stable Snaith splitting of the double loop space of a Moore space and use it to construct infinite $v_1$-periodic families of elements of order $p^{r+1}$ in the homotopy groups of mod $p^r$ Moore spaces. For odd primes $p$, our splitting implies that the homotopy groups of the mod $p^{r+1}$ Moore spectrum are summands of the unstable homotopy groups of each mod $p^r$ Moore space.
\end{abstract}

\maketitle 

\section{Introduction}

The purpose of this note is to combine three standard results in homotopy theory:
\begin{enumerate}[label=(\arabic*)]
\item the construction of elements of order $p^{r+1}$ in the homotopy groups of the mod $p^r$ Moore space $P^n(p^r)$, as described in \cite{CMN1};
\item the stable splitting of $\Omega^2P^n(p^r)$ first proved by Snaith \cite{S}; and
\item the introduction of $v_1$-periodic self-maps by Adams in his work on the image of the $J$-homomorphism. 
\end{enumerate}

In their fundamental work on the homotopy theory of Moore spaces, Cohen, Moore and Neisendorfer \cite{CMN1,CMN2, N, N1} proved that $P^n(p^r)=S^{n-1}\cup_{p^r}e^n$ has homotopy exponent exactly $p^{r+1}$ when $p$ is a prime number greater than $3$. We will refer to elements of this maximal possible order in $\pi_\ast(P^n(p^r))$ as \emph{higher torsion} elements. The main results of this paper give additional infinite families of higher torsion elements in the homotopy groups of odd primary Moore spaces which are different from
those constructed via Samelson products in \cite{CMN1}. The main technical ingredient is a slightly finer stable decomposition of $\Omega^2P^{2n+1}(p^r)$ which essentially follows from a combination of results (1) and (2) above.
The reason that this question arose is because of computations of Roman Mikhailov and Jie Wu, who asked about ``functorial elements"  of order $p^{r+1}$ in the homotopy groups of mod $p^r$ Moore spaces not given by those in \cite{CMN1}.

By considering the integral homology of the double loop space of a Moore space, it is clear that certain spherical homology classes force the classical Snaith splitting of $\Omega^2P^{2n+1}(p^r)$ to stably decompose further then previously described. This new stable splitting allows for the construction of new higher torsion elements which are detected by $K$-theory but not detected in the ordinary homology of any iterated loop space of a Moore space, unlike the elements of order $p^{r+1}$ in \cite{CMN1} which have non-trivial Hurewicz images in the homology of $\Omega^2P^{2n+1}(p^r)$ (see Lemma~\ref{odd higher Bockstein}).

The main results are described next. Recall that the Snaith splitting gives a functorial stable homotopy equivalence
\[ \Sigma^\infty\Omega^2 \Sigma^2X \simeq \Sigma^\infty\bigvee_{j=1}^\infty D_j(\Omega^2\Sigma^2X) \]
for any path-connected CW-complex $X$, where the stable summands are given by suspension spectra of the extended powers 
$D_j(\Omega^2\Sigma^2X) = \mathcal{C}_2(j)_+ \wedge_{\Sigma_j} X^{\wedge j}$
and $\mathcal{C}_2(j)$ denotes the space of~$j$ little $2$-cubes disjointly embedded in $\mathbb{R}^2$. 
In the case that $X$ is an odd-dimensional sphere $S^{2n-1}$, the stable summands $D_j(\Omega^2S^{2n+1})$ of $\Omega^2S^{2n+1}$ have been well studied; they are $p$-locally contractible unless $j \equiv 0$ or $1$ mod $p$, in which case they can be identified with suitably suspended Brown--Gitler spectra. In particular, after localizing at a prime $p$, they are stably indecomposable. Below we consider the case of an odd-dimensional Moore space and the stable summands $D_{p^k}(\Omega^2P^{2n+1}(p^r))$ which map naturally onto these Brown--Gitler spectra by the map $\Omega^2\Sigma^2q$ where $q\colon P^{2n-1}(p^r) \to S^{2n-1}$ is the pinch map.

\begin{thm} \label{splitting thm}
Suppose $p$ is prime and $n>1$.
\begin{enumerate}
\item If $p\ge 3$ and $r \ge 1$, then $D_{p^k}(\Omega^2P^{2n+1}(p^r))$ is stably homotopy equivalent to
\[ P^{2np^k-2}(p^{r+1}) \vee  X_{p^k} \]
for some finite CW-complex $X_{p^k}$ for all $k\ge 1$. 
\item If $p=2$ and $r>1$, then $D_2(\Omega^2P^{2n+1}(2^r))$ is homotopy equivalent to 
\[ P^{4n-2}(2^{r+1}) \vee  X_2 \] 
for some $4$-cell complex $X_2=P^{4n-3}(2^r) \cup CP^{4n-2}(2)$.
\item If $p=2$ and $r=1$, then $D_2(\Omega^2P^{2n+1}(2))$ is a stably indecomposable $6$-cell complex.
\end{enumerate}
\end{thm}

The reason the stable splitting of $\Omega^2P^{2n+1}(p^r)$ described by Theorem \ref{splitting thm} has implications for the unstable homotopy groups of $P^{2n+1}(p^r)$ is that maps $P^{2np^k-2}(p^{r+1}) \to \Omega^2P^{2n+1}(p^r)$ admitting stable retractions exist unstably when $p$ is odd (and in only a few cases when $p=2$; see Section \ref{new families section}).

\begin{thm} \label{factorization thm}
Let $p$ be an odd prime, $r\ge 1$ and $n>1$. Then for every $k\ge 1$ there exist homotopy commutative diagrams
\[
\xymatrix{
P^{2np^k-2}(p^{r+1}) \ar[r] \ar[rd]_{E^\infty} & \Omega^2P^{2n+1}(p^r) \ar[d] \\
& QP^{2np^k-2}(p^{r+1})
} \qquad\qquad
\xymatrix @C=1pc{
P^{(4n-2)p^k-2}(p^{r+1}) \ar[r] \ar[rd]_{E^\infty} & \Omega^2P^{2n}(p^r) \ar[d] \\
& QP^{(4n-2)p^k-2}(p^{r+1})
}
\]
where $E^\infty$ is the stabilization map (i.e., unit of the adjunction $\Sigma^\infty \dashv \Omega^\infty$).
\end{thm}

The loop space decompositions of odd primary Moore spaces given in \cite{CMN1,N1} imply that the stable homotopy groups $\pi_\ast^s(P^n(p^r))$ are in a certain sense retracts of the unstable homotopy groups $\pi_\ast(P^n(p^r))$. Different loop space decompositions were used in \cite{CW} to obtain the same result for $2$-primary Moore spaces, and other examples of spaces whose stable and unstable homotopy groups share this property are given in \cite{BW}. As a consequence of Theorem~\ref{factorization thm}, the stable homotopy groups of the mod $p^{r+1}$ Moore spectrum retract off the unstable homotopy groups of each mod $p^r$ Moore space in a similar sense when $p$ is an odd prime (see Corollary~\ref{stable unstable}).

This observation clearly suggests that $\pi_\ast(P^n(p^r))$ contains many $\mathbb{Z}/p^{r+1}$ summands when~$p$ is odd. To generate explicit examples, in Section \ref{new families section} we use desuspensions of Adams maps in conjunction with Theorem~\ref{factorization thm} to construct infinite $v_1$-periodic families of higher torsion elements and obtain the following. Let $q=2(p-1)$.

\begin{thm} \label{new odd families}
Let $p$ be an odd prime, $r\ge 1$ and $n>1$. Then for all sufficiently large $k$,
\[ \pi_{2np^k-1+tqp^r}(P^{2n+1}(p^r)) \, \text{ and }\  \pi_{(4n-2)p^k-1+tqp^r}(P^{2n}(p^r)) \] 
contain $\mathbb{Z}/p^{r+1}$ summands for every $t\ge 0$. 
\end{thm}

\begin{rem}
A lower bound on $k$ in Theorem \ref{new odd families} is required to ensure the existence of unstable Adams maps $v_1\colon P^{\ell+qp^r}(p^{r+1}) \to P^\ell(p^{r+1})$ which induce isomorphisms in $K$-theory. See Section \ref{new families section} for a more precise statement. In particular, in the most interesting case when $r=1$, we only require $k\ge 1$. 
\end{rem}

For $p=2$, unstable maps analogous to those in Theorem~\ref{factorization thm} rarely exist for reasons related to the divisibility of the Whitehead square. When $r>1$, the inclusion of the stable summand $P^{4n-2}(2^{r+1})$ of $\Omega^2P^{2n+1}(2^r)$ given by Theorem~\ref{splitting thm}(b) exists unstably if and only if $n=2$ or $4$. As examples in these cases, we describe unstable $v_1$-periodic families of higher torsion elements in $\pi_\ast(P^5(2^r))$ and $\pi_\ast(P^9(2^r))$ for some small values of $r$ (Theorem~\ref{new even families}). 


\begin{rem}
When $p$ is odd, the $r>1$ case of Theorem~\ref{splitting thm}(a) follows quickly from an unstable product decomposition of $\Omega^2P^{2n+1}(p^r)$ proved by Neisendorfer. More precisely, \cite[Theorem~1]{N} shows that $\Omega S^{2np^k-1}\{p^{r+1}\}$ is a retract of $\Omega^2P^{2n+1}(p^r)$ for all $k\ge 1$, and it is readily checked that $P^{2np^k-2}(p^{r+1})$ is a stable retract of $\Omega S^{2np^k-1}\{p^{r+1}\}$ (cf.\ \cite[Proposition~4.1]{Am}). The $r=1$ case of Theorem~\ref{splitting thm}(a) above gives some evidence for conjectures surrounding the unstable homotopy type of $\Omega^2P^{2n+1}(p)$ considered in \cite{CMN2,G2,N double,T}.
\end{rem}

\begin{ack}
This material is based upon work partly supported by the National Science Foundation under Grant No.\ DMS-1929284 while the first author was in residence at the Institute for Computational and Experimental Research in Mathematics in Providence, RI.
\end{ack}

\section{Homology of the stable summands $D_j(\Omega^2P^{2n+1}(p^r))$}

The homology of $\Omega^n\Sigma^nX$ taken with field coefficients
as a filtered algebra was worked out in \cite{CLM}. A short summary of that information elucidates the homology of the stable Snaith summands, usually denoted $D_{n,j}X$. In the applications below where $n=2$, the $j^\mathrm{th}$ stable summand $D_{2,j}X$ of $\Omega^2\Sigma^2X$ will be denoted $D_j(\Omega^2\Sigma^2X)$. All homology groups have $\mathbb{Z}/p$ coefficients unless indicated otherwise.

Start with the connected graded vector space $$V=\bar{H}_\ast(X)$$ given by the reduced mod~$p$ homology of a path-connected space~$X$. Next, consider the reduced homology of the suspension $\Sigma X$,
denoted $\sigma V$. Form the free graded Lie algebra generated by $\sigma V$, denoted \[L[\sigma V].\]
In addition, consider the free graded restricted Lie algebra \[L^p[\sigma V].\]
This restricted Lie algebra is isomorphic to the module of primitive elements in the tensor algebra generated by $\sigma V$, so the tensor algebra is a primitively generated Hopf algebra.

A basis for $L^p[\sigma V]$ is given by the union of:
\begin{enumerate}[label=(\arabic*)]
\item a basis $\mathcal B =\{b_{\alpha} \mid \alpha \in I\}$ for $ L[\sigma(V)]_{odd}$, the elements of odd degree in $L[\sigma(V)]$;
\item a basis $\mathcal C = \{c_{\gamma} \mid \gamma \in J\}$ for $L[\sigma(V)]_{even}$, the elements of even degree in $L[\sigma(V)]$; and
\item a basis for the $(p^k)^\mathrm{th}$ powers of $L[\sigma(V)]_{even}$, say
$\mathcal P\mathcal C =  \{{c_{\gamma}}^{p^k} \mid \gamma \in J,\; k\ge 1\}$.
\end{enumerate}
It follows from the Bott--Samelson theorem that a basis for the module of primitives in the mod~$p$ homology of $\Omega \Sigma^2X$ is given by
\[ {\mathcal B} \cup {\mathcal C} \cup { \mathcal P\mathcal C}. \]

The mod $p$ homology of $\Omega^2\Sigma^2X$ can now be described using the preparations of the previous paragraph.
First, for each $(p^k)^\mathrm{th}$ power $x = {c_{\gamma}}^{p^k} \in \mathcal P \mathcal C$, let $\sigma^{-1}x$ denote the formal desuspension, lowering the degree of $x$ by one. Let $\beta\sigma^{-1}x$ denote the formal first Bockstein of $\sigma^{-1}{c_{\gamma}}^{p^k}$, with degree $|x|-2$. Let $\Psi$ denote the set of elements
given by \[ \Psi = \{ \sigma^{-1}x, \beta\sigma^{-1}x \mid x\in \mathcal P \mathcal C \}. \]

\begin{thm}[\cite{CLM}] \label{CLM thm}
If $X$ is a path-connected CW-complex and $p$ is an odd prime, then the mod~$p$ homology $H_\ast(\Omega^2\Sigma^2X)$ is isomorphic as an algebra to the free graded commutative algebra generated by
\[ \sigma^{-1} {\mathcal B} \cup \sigma^{-1} {\mathcal C} \cup \Psi. \]
\end{thm}

\begin{rem}
The result in case $p=2$ is different and mildly simpler (see \cite{CLM} and the remarks preceeding Lemma~\ref{Steenrod squares} below).
\end{rem}

The homology of iterated loop-suspensions are free $E_n$-algebras naturally equipped with more algebraic structure than we will need here, but we briefly mention two homology operations which will often be used below to label elements of $H_\ast(\Omega^2\Sigma^2X)$ when $X=P^{2n-1}(p^r)$; namely, the Dyer--Lashof operation 
\[ Q_1\colon H_n(\Omega^2\Sigma^2X) \longrightarrow H_{np+p-1}(\Omega^2\Sigma^2X) \] 
and the Browder bracket 
\[ \lambda\colon H_n(\Omega^2\Sigma^2X) \otimes H_m(\Omega^2\Sigma^2X) \longrightarrow H_{n+m+1}(\Omega^2\Sigma^2X). \] 
The images of $Q_1$ and $\lambda$ contain the transgressions of~$(p^k)^\mathrm{th}$ powers and iterated commutators, respectively, of primitive elements in the tensor Hopf algebra $H_\ast(\Omega\Sigma^2X)$. See \cite{CLM} for a description of the generators in $\sigma^{-1} {\mathcal B} \cup \sigma^{-1} {\mathcal C} \cup \Psi$ in terms of these operations.

We recall that the set of indecomposables $\sigma^{-1} {\mathcal B} \cup \sigma^{-1} {\mathcal C} \cup \Psi$ is also graded by weights as in \cite{CLM} in such a way that the reduced mod $p$ homology of $D_j(\Omega^2\Sigma^2X)$ is spanned by the monomials of weight $j$ in the free graded commutative algebra generated by $\sigma^{-1} {\mathcal B} \cup \sigma^{-1} {\mathcal C} \cup \Psi$.

Explicitly, for $X=P^{2n-1}(p^r)$, let $\{u,v\}$ be a basis for the graded vector space $V=\bar{H}_\ast(P^{2n-1}(p^r))$ with $|u|=2n-2$, $|v|=2n-1$ and $\beta^{(r)}v=u$. Then, if $p$ is odd, Theorem~\ref{CLM thm} implies that $H_\ast(\Omega^2P^{2n+1}(p^r))$ is a free graded commutative algebra on generators
\begin{alignat*}{2}
&\left. u,\ v,\  \lambda(u,u),\  \lambda(u,v),\  \lambda(u,\lambda(u,v)),\  \lambda(v, \lambda(u,v)), \ldots \right. &&\in\sigma^{-1}{\mathcal B}\cup\sigma^{-1}{\mathcal C} \\
&\left.\begin{aligned}
&Q_1v,\  \beta^{(1)}Q_1v,\  Q_1^2v,\  \beta^{(1)}Q_1^2v, \ldots & \\
&Q_1\lambda(u,u),\  \beta^{(1)}Q_1\lambda(u,u),\  Q_1^2\lambda(u,u),\  \beta^{(1)}Q_1^2\lambda(u,u), \ldots & \\
&Q_1\lambda(u,\lambda(u,v)),\  \beta^{(1)}Q_1\lambda(u,\lambda(u,v)), \ldots & \\
&\vdotswithin{Q_1}
\end{aligned} \right\} &&\in \Psi
\end{alignat*}
with weights defined by  
\begin{align*}
\mathrm{wt}(u) &= \mathrm{wt}(v)=1, \\ 
\mathrm{wt}(Q_1^kx) &= \mathrm{wt}(\beta^{(1)}Q_1^kx) = p^k\mathrm{wt}(x), \\
\mathrm{wt}(\lambda(x,y)) &= \mathrm{wt}(x)+\mathrm{wt}(y),
\end{align*}
and extended to all monomials by $\mathrm{wt}(xy)=\mathrm{wt}(x)+\mathrm{wt}(y)$.

We will often use the same notation $u$, $v$ to denote generators of the bottom two mod $p$ homology groups of $P^{2n+1}(p^r)$, $\Omega P^{2n+1}(p^r)$ and $\Omega^2P^{2n+1}(p^r)$, indicating degrees with subscripts when necessary.

To prove the splittings in Theorem~\ref{splitting thm}, we will need to know the top two homology groups of the Snaith summands $D_{p^k}(\Omega^2P^{2n+1}(p^r))$ explicitly. Although the list of generators above is specific to the case of $p$ odd, the following lemma holds for all primes $p$.

\begin{lem} \label{top homology}
For each $k\ge 0$, $D_{p^k}(\Omega^2P^{2n+1}(p^r))$ is a $(2np^k-2p^k-1)$-connected $(2np^k-1)$-dimensional space with:
\begin{enumerate}
\item $H_{2np^k-1}(D_{p^k}(\Omega^2P^{2n+1}(p^r))) = \mathrm{span}\{ Q_1^kv \}$,
\item $H_{2np^k-2}(D_{p^k}(\Omega^2P^{2n+1}(p^r))) = \mathrm{span}\{ \beta^{(1)}Q_1^kv, \; \mathrm{ad}_\lambda^{p^k-1}(v)(u) \}$, \label{part b}
\end{enumerate}
where $\mathrm{ad}_\lambda^{p^k-1}(v)(u)$ denotes the $(p^k-1)$-fold iterated Browder bracket $\lambda(v, \lambda(v, \ldots, \lambda(v,u)\ldots ))$.
\end{lem}

\begin{proof}
Since the reduced mod $p$ homology of $D_{p^k}(\Omega^2P^{2n+1}(p^r))$ consists of the elements of homogeneous weight $p^k$ in $H_\ast(\Omega^2P^{2n+1}(p^r))$, the connectivity and dimension of $D_{p^k}(\Omega^2P^{2n+1}(p^r))$ follow from the fact that the weight $p^k$ monomials of lowest and highest homological degree are $u^{p^k}$ and $Q_1^kv$, respectively, with $|u^{p^k}|=2np^k-2p^k$ and $|Q_1^kv|=2np^k-1$.

Observe that any nonzero iterated Browder bracket with arguments in $\{u,v\}$ must involve $u\in H_{2n-2}(\Omega^2P^{2n+1}(p^r))$ since $\lambda(v,v)=0$ (being the transgression of the graded commutator of an even degree element with itself in the tensor algebra $H_\ast(\Omega P^{2n+1}(p^r))$). Parts (a) and (b) now follow easily by inspection of monomials of weight $p^k$ in homological degrees $2np^k-1$ and $2np^k-2$.
\end{proof}

\begin{rem}
Note that for $k=0$, the span in Lemma~\ref{top homology}\ref{part b} is $1$-dimensional since $\beta^{(1)}v$ and $\mathrm{ad}_\lambda^0(v)(u)$ coincide if $r=1$, and $\beta^{(1)}v=0$ if $r>1$. Of course, in this case $D_{p^k}(\Omega^2P^{2n+1}(p^r))=D_1(\Omega^2P^{2n+1}(p^r))$ is simply $P^{2n-1}(p^r)$. For all $k\ge 1$, $\dim H_{2np^k-2}(D_{p^k}(\Omega^2P^{2n+1}(p^r))) = 2$.
\end{rem}

In the $p=2$ case we will need to know the homology of $D_2(\Omega^2P^{2n+1}(2^r))$ as a module over the Steenrod algebra. The mod $2$ homology generators of weight~$2$ differ somewhat from those appearing in the list above for odd primes. First, since $H_\ast(\Omega^2P^{2n+1}(2^r))$ is a polynomial algebra, we have the quadratic generator $v^2$ in addition to $u^2$ and $uv$. Second, the Browder bracket $\lambda(u,u)$ is trivial since this class represents the transgression of the commutator $[u,u]=u^2+u^2=0$ in the tensor algebra $H_\ast(\Omega P^{2n+1}(2^r))=T(u,v)$ over~$\mathbb{Z}/2$. On the other hand, since $u^2\in H_{4n-2}(\Omega P^{2n+1}(2^r))$ is primitive, this class transgresses to a generator $Q_1u$ in $H_{4n-3}(\Omega^2P^{2n+1}(2^r))$, unlike in the odd primary case.

It follows that $D_2(\Omega^2P^{2n+1}(2^r))$ is a $6$-cell complex with
\begin{align*}
H_{4n-1}(D_2(\Omega^2P^{2n+1}(2^r))) &= \mathrm{span}\{ Q_1v \}, \\
H_{4n-2}(D_2(\Omega^2P^{2n+1}(2^r))) &= \mathrm{span}\{ v^2, \; \lambda(u,v) \}, \\
H_{4n-3}(D_2(\Omega^2P^{2n+1}(2^r))) &= \mathrm{span}\{ uv, \; Q_1u \}, \\
H_{4n-4}(D_2(\Omega^2P^{2n+1}(2^r))) &= \mathrm{span}\{ u^2 \}.
\end{align*}

\begin{lem} \label{Steenrod squares}
The action of the Steenrod algebra on $H_\ast(D_2(\Omega^2P^{2n+1}(2^r)))$ is determined by:
\begin{enumerate}
\item $Sq^1_\ast Q_1v = \begin{cases} v^2+\lambda(u,v) & \text{if } r=1 \\ v^2 & \text{if } r>1, \end{cases}$ \\
\item $Sq^2_\ast Q_1v = \begin{cases} Q_1u & \text{if } r=1 \\ 0 & \text{if } r>1, \end{cases}$ \\
\item $Sq^2_\ast v^2 = Sq^1_\ast uv = \begin{cases} u^2 & \text{if } r=1 \\ 0 & \text{if } r>1. \end{cases}$ \\
\end{enumerate}
\end{lem}

\begin{proof}
By \cite[III.3.10]{CLM}, $\beta^{(1)}Q_1x=x^2+\lambda(x,\beta^{(1)}x)$ for $x\in H_\ast(\Omega^2\Sigma^2X;\mathbb{Z}/2)$ with $|x|$ odd. Part (a) follows since $\beta^{(1)}v=u$ if $r=1$, and $\beta^{(1)}v=0$ if $r>1$. 

Part (b) follows from the Nishida relation $Sq^2_\ast Q_1=Q_1Sq^1_\ast$.

For part (c), the Cartan formula and the fact that $Sq^1_\ast=\beta^{(1)}$ is a derivation on the Pontryagin ring $H_\ast(\Omega^2P^{2n+1}(2^r))$ imply that $Sq^2_\ast v^2$ and $Sq^1_\ast uv$ are as claimed.

Since the Browder bracket satisfies the Cartan formula \cite[III.1.2(7)]{CLM} 
\[ Sq^n_\ast \lambda(x,y) = \sum_{i+j=n}\lambda(Sq^i_\ast x,Sq^j_\ast y), \]
we have $Sq^2_\ast \lambda(u,v) = Sq^1_\ast \lambda(u,v) = 0$. The relations above therefore determine all nontrivial Steenrod operations in $H_\ast(D_2(\Omega^2P^{2n+1}(2^r)))$.
\end{proof}

\section{Review of the work of Cohen, Moore and Neisendorfer}

To prepare for the proofs of Theorems \ref{splitting thm} and \ref{factorization thm} in the next section, we briefly review some of the work of Cohen, Moore and Neisendorfer \cite{CMN1, N} on torsion in the homotopy groups of Moore spaces.

Recall that the mod $p^r$ homotopy groups of a space $X$ are defined by $\pi_n(X;\mathbb{Z}/p^r) = [P^n(p^r),X]$. Provided $p^r>2$, there are splittings \cite[Proposition 6.2.2]{N2}
\[ P^n(p^r) \wedge P^m(p^r) \simeq P^{n+m}(p^r) \vee P^{n+m-1}(p^r) \] 
for $n,m\ge 2$ which allow for the definition of mod $p^r$ Samelson products 
\[ \pi_n(\Omega X; \mathbb{Z}/p^r) \otimes \pi_m(\Omega X; \mathbb{Z}/p^r) \longrightarrow \pi_{n+m}(\Omega X; \mathbb{Z}/p^r). \]
Together with the Bockstein differential, this gives $\pi_\ast(\Omega X; \mathbb{Z}/p^r)$ the structure of a differential graded Lie algebra when $p\ge 5$, $r\ge 1$ and a differential graded quasi-Lie algebra when $p=3$, $r\ge 2$ (see \cite{N, N2}).
The mod $p$ Hurewicz map \[ h\colon \pi_\ast(\Omega X;\mathbb{Z}/p) \longrightarrow H_\ast(\Omega X;\mathbb{Z}/p) \] intertwines mod $p$ Samelson products with commutators in the Pontryagin ring and commutes with Bockstein differentials, thereby inducing a morphism of spectral sequences from the mod~$p$ homotopy Bockstein spectral sequence $(E_\pi^s(\Omega X), \beta^{(s)})$ to the mod~$p$ homology Bockstein spectral sequence $(E_H^s(\Omega X), \beta^{(s)})$.

Consider $\pi_\ast(\Omega P^{2n+1}(p^r);\mathbb{Z}/p)$. In degrees $2n$ and $2n-1$, denote the mod $p$ reduction of the adjoint of the identity map and its $r^\text{th}$ Bockstein by
\begin{align*} 
\nu &\colon P^{2n}(p) \longrightarrow \Omega P^{2n+1}(p^r) \;\text{ and} \\
\beta^{(r)}\nu=\mu &\colon P^{2n-1}(p) \longrightarrow \Omega P^{2n+1}(p^r),
\end{align*}
respectively. Then the Hurewicz images $h(\nu)=v$ and $h(\mu)=u$ generate $H_{2n}(\Omega P^{2n+1}(p^r);\mathbb{Z}/p)$ and $H_{2n-1}(\Omega P^{2n+1}(p^r);\mathbb{Z}/p)$, respectively, and by the Bott--Samelson theorem, 
\begin{equation} \label{tensor alg}
H_\ast(\Omega P^{2n+1}(p^r); \mathbb{Z}/p) \cong T(u,v) \cong UL(u,v), 
\end{equation} 
where $L(u,v)$ is the free differential graded Lie algebra on two generators $u$, $v$ with differential~$\beta^{(r)}v=u$.

In any graded (quasi-)Lie algebra $L$ (more generally, any graded module with an antisymmetric bracket operation), let $\mathrm{ad}(x)(y)=[x,y]$ for $x,y\in L$. Define $\mathrm{ad}^0(x)(y)=y$ and inductively define $\mathrm{ad}^k(x)(y)=\mathrm{ad}(x)(\mathrm{ad}^{k-1}(x)(y))$ for $k\ge 1$. 
To detect higher torsion in $\pi_\ast(P^{2n+1}(p^r))$, Cohen, Moore and Neisendorfer \cite{CMN1} consider the mod $p$ Samelson products
\begin{align*}
\tau_k(\nu) &= \mathrm{ad}^{p^k-1}(\nu)(\mu), \\
\sigma_k(\nu) &= \frac{1}{2p}\sum_{j=1}^{p^k-1}\binom{p^k}{j}[\mathrm{ad}^{j-1}(\nu)(\mu), \mathrm{ad}^{p^k-j-1}(\nu)(\mu)]
\end{align*}
in $\pi_\ast(\Omega P^{2n+1}(p^r); \mathbb{Z}/p)$ and their mod $p$ Hurewicz images $\tau_k(v)$, $\sigma_k(v)$ defined similarly in terms of graded commutators. Since the tensor algebra \eqref{tensor alg} is acyclic with respect to the differential $\beta^{(r)}$, the homology Bockstein spectral sequence collapses at the~$(r+1)^\text{st}$ page and no higher differentials in the homotopy Bockstein spectral sequence can be detected by the Hurewicz map:
\[ E_H^1(\Omega P^{2n+1}(p^r))=\cdots=E_H^r(\Omega P^{2n+1}(p^r))=T(u,v), \quad \bar{E}_H^{r+1}(\Omega P^{2n+1}(p^r))=0. \]
In particular, $\tau_k(v)\in H_{2np^k-1}(\Omega P^{2n+1}(p^r))$ is killed by the differential $\beta^{(r)}v^{p^k}=\tau_k(v)$ for all~$k\ge 0$. To tease out higher torsion, Cohen, Moore and Neisendorfer instead compute the homology Bockstein spectral sequence of the loops on the fibre $F^{2n+1}(p^r)$ of the pinch map~$q\colon P^{2n+1}(p^r) \to S^{2n+1}$, where lifts $\tau'_k(\nu)$, $\sigma'_k(\nu) \in \pi_\ast(\Omega F^{2n+1}(p^r); \mathbb{Z}/p)$ of $\tau_k(\nu)$, $\sigma_k(\nu)$ and their Hurewicz images $\tau'_k(v)$, $\sigma'_k(v)$ are shown to survive to the $(r+1)^\text{st}$ page, at least when~$p$ is odd.

\begin{thm}[{\cite[Theorem~10.3]{CMN1}}] \label{Bockstein in fibre}
Let $p$ be an odd prime and $r\ge 1$. Then there is an isomorphism of differential graded Hopf algebras
\[ E_H^{r+1}(\Omega F^{2n+1}(p^r)) \cong \Lambda(\tau'_0(v), \tau'_1(v), \tau'_2(v),\ldots) \otimes \mathbb{Z}/p\left[\sigma'_1(v),\sigma'_2(v),\ldots\right] \]
where $|\tau'_k(v)|=2np^k-1$, $|\sigma'_k(v)|=2np^k-2$ and \[ \beta^{(r+1)}\tau'_k(v) = \ell\sigma'_k(v), \ \ell\neq 0,\] for $k\ge 1$. 
\end{thm}

Since the homology classes $\tau'_k(v)=h(\tau'_k(\nu))$ support nontrivial Bocksteins $\beta^{(r+1)}$ for $k\ge 1$, the same is true of the mod $p$ homotopy classes $\tau'_k(\nu) \in \pi_{2np^k-1}(\Omega F^{2n+1}(p^r); \mathbb{Z}/p)$. It follows that there exist maps
\begin{equation} \label{delta' map}
\delta'_k \colon P^{2np^k-1}(p^{r+1}) \longrightarrow \Omega F^{2n+1}(p^r)
\end{equation}
for each $k\ge 1$ which satisfy $(\delta'_k)_\ast(v_{2np^k-1}) = \tau'_k(v)$ in mod $p$ homology and induce split monomorphisms in integral homology.

To exhibit nontrivial classes in the image of $\beta^{(r+1)}$ in $E_\pi^{r+1}(\Omega P^{2n+1}(p^r))$, Cohen, Moore and Neisendorfer show that the composition of $\beta^{(r+1)}\tau'_k(\nu)$ with $\Omega F^{2n+1}(p^r) \to \Omega P^{2n+1}(p^r)$ does not represent zero in $E_\pi^{r+1}(\Omega P^{2n+1}(p^r))$, thereby proving the following.

\begin{thm}[\cite{CMN1,N}] \label{CMN torsion}
Let $p$ be an odd prime and $r\ge 1$. Then $\pi_{2np^k-1}(P^{2n+1}(p^r))$ contains a $\mathbb{Z}/p^{r+1}$ summand for every $k\ge 1$.
\end{thm}

\begin{rem}
It follows from Theorem \ref{CMN torsion} and the loop space decomposition for even-dimensional odd primary Moore spaces \cite[Theorem~1.1]{CMN1}
\begin{equation} \label{even decomposition}
\Omega P^{2n+2}(p^r) \simeq S^{2n+1}\{p^r\} \times \Omega \left( \bigvee_{j=0}^\infty P^{4n+2nj+3}(p^r) \right)
\end{equation}
that $\pi_\ast(P^n(p^r))$ contains $\mathbb{Z}/p^{r+1}$ summands for all $n\ge 3$ when $p$ is odd.
\end{rem}

\section{Splittings of $D_{p^k}(\Omega^2P^{2n+1}(p^r))$}

In this section we prove Theorem \ref{splitting thm} in a series of lemmas and discuss the stable homotopy type of $\Omega^2P^{2n+1}(p^r)$. We assume throughout that $n>1$.

The higher torsion discussed in the previous section is not reflected in the homology of the single loop space of a Moore space since $H_\ast(\Omega P^n(p^r))$ is acyclic with respect to $\beta^{(r)}$. The next lemma shows that it becomes visible in homology after looping twice.

\begin{lem} \label{odd higher Bockstein}
Let $p$ be an odd prime and $r\ge 1$. Then in the mod $p$ homology Bockstein spectral sequence of $\Omega^2P^{2n+1}(p^r)$, we have 
\[ \beta^{(r+1)} \mathrm{ad}_\lambda^{p^k-1}(v)(u) \neq 0 \]
in $E_H^{r+1}(\Omega^2P^{2n+1}(p^r))$ for each $k\ge 1$. Moreover, there exist maps 
\[ \delta_k \colon P^{2np^k-2}(p^{r+1}) \longrightarrow \Omega^2P^{2n+1}(p^r) \]
for each $k\ge 1$ which satisfy $(\delta_k)_\ast(v_{2np^k-2}) = \mathrm{ad}_\lambda^{p^k-1}(v)(u)$ in mod $p$ homology and induce split monomorphisms in integral homology.
\end{lem}

\begin{proof}
Consider the composite 
\[ P^{2np^k-1}(p^{r+1}) \stackrel{\delta'_k}{\longrightarrow} \Omega F^{2n+1}(p^r) \longrightarrow \Omega P^{2n+1}(p^r), \]
where $\delta'_k$ is the map from \eqref{delta' map} and the second map is the fibre inclusion of the looped pinch map $\Omega q\colon \Omega P^{2n+1}(p^r) \to \Omega S^{2n+1}$. In mod $p$ homology, $(\delta'_k)_\ast(v_{2np^k-1})=\tau'_k(v)$ and $(\delta'_k)_\ast(u_{2np^k-2})=\ell\sigma'_k(v)$, $\ell\neq 0$, by Theorem~\ref{Bockstein in fibre} and naturality of the Bockstein. Since these classes map to $\tau_k(v)$, $\ell\sigma_k(v) \in H_\ast(\Omega P^{2n+1}(p^r))$, the composite above induces a monomorphism in mod~$p$ homology (with Bocksteins acting trivially on the image).

Define $\delta_k \colon P^{2np^k-2}(p^{r+1}) \to \Omega^2P^{2n+1}(p^r)$ to be the adjoint of the composite above. Let $\tau^\lambda_k(v)$ denote the iterated Browder bracket
\[ \tau^\lambda_k(v) = \mathrm{ad}_\lambda^{p^k-1}(v)(u) \in H_{2np^k-2}(\Omega^2P^{2n+1}(p^r)) \]
which is the transgression of the iterated commutator
\[ \tau_k(v)=\mathrm{ad}^{p^k-1}(v)(u) \in H_{2np^k-1}(\Omega P^{2n+1}(p^r)). \]
It follows that $(\delta_k)_\ast(v_{2np^k-2}) = \tau^\lambda_k(v)$. Similarly, $(\delta_k)_\ast(u_{2np^k-3}) = \ell\sigma^\lambda_k(v) \neq 0$ where $\sigma^\lambda_k(v)$ denotes the transgression of $\sigma_k(v) \in H_{2np^k-2}(\Omega P^{2n+1}(p^r))$.

We now have a map $\delta_k\colon P^{2np^k-2}(p^{r+1}) \to \Omega^2P^{2n+1}(p^r)$ inducing
\[
\xymatrix @R=1.2pc @C=5pc{
\quad v_{2np^k-2} \ar@/_1pc/[d]_-{\beta^{(r+1)}} \ar@{|->}[r] & \tau^\lambda_k(v) \qquad \\
\quad u_{2np^k-3} \ar@{|->}[r] & \ell\sigma^\lambda_k(v) \qquad
}
\]
in homology and it remains to show that $\beta^{(r+1)}\tau^\lambda_k(v) \neq 0$. It suffices by naturality of $\beta^{(r+1)}$ to show that $\sigma^\lambda_k(v)$ does not represent zero in $E_H^{r+1}(\Omega^2P^{2n+1}(p^r))=H_\ast(E_H^r(\Omega^2P^{2n+1}(p^r)), \beta^{(r)})$. First note that $\tau^\lambda_k(v)$ and $\sigma^\lambda_k(v)$ are $\beta^{(s)}$-cycles for $s\le r$ since $v_{2np^k-2}$ and $u_{2np^k-3}$ are. To see that they are not $\beta^{(s)}$-boundaries for any $s\le r$, consider the Snaith splitting of $\Omega^2P^{2n+1}(p^r)$. Since all Bocksteins must respect the induced splitting in homology and $\tau^\lambda_k(v)$, $\sigma^\lambda_k(v)$ lie in the homology of the stable summand $D_{p^k}(\Omega^2P^{2n+1}(p^r))$, it follows from Lemma~\ref{top homology} that for degree reasons the only class $x$ which could potentially satisfy $\beta^{(s)}x=\sigma^\lambda_k(v)$ is a linear combination of $\beta^{(1)}Q_1^kv$ and $\tau^\lambda_k(v)$. But $\beta^{(s)}\beta^{(1)}Q_1^kv=\beta^{(s)}\tau^\lambda_k(v)=0$ for all $s\le r$. Therefore $\tau^\lambda_k(v)$, $\sigma^\lambda_k(v)$ represent nontrivial classes in $E_H^{r+1}(\Omega^2P^{2n+1}(p^r))$, where the differential $\beta^{(r+1)}\tau^\lambda_k(v) = \ell\sigma^\lambda_k(v)$ is forced.
\end{proof}

As a partial $2$-primary analogue of Lemma~\ref{odd higher Bockstein}, we show that the class $\mathrm{ad}_\lambda^{2^k-1}(v)(u)$ supports a higher Bockstein when $k=1$.

\begin{lem} \label{even higher Bockstein} 
Let $p=2$ and $r\ge 1$. Then in the mod $2$ homology Bockstein spectral sequence of $\Omega^2P^{2n+1}(2^r)$, we have
\[ \beta^{(r+1)}\lambda(u,v) = Q_1u \]
in $E_H^{r+1}(\Omega^2P^{2n+1}(2^r))$.
\end{lem}

\begin{proof}
We give a direct chain level calculation similar to the proof of \cite[III.3.10]{CLM}. Consider the $\Sigma_2$-invariant map $\theta\colon\mathcal{C}_2(2)\times X\times X \to X$ given by the action of the little $2$-cubes operad on $X=\Omega^2P^{2n+1}(2^r)$. Let $e_k$ and $\alpha$ be as defined in \cite[Section 6]{M} and let $a$, $b$ be chains representing $v$, $u$, respectively, with $d(a)=2^rb$. Then
\begin{align*}
d((\alpha +1)e_1\otimes a\otimes b) &= (\alpha^2-1)e_0\otimes a\otimes b - 2^r(\alpha+1)e_1\otimes b\otimes b \\
&= - 2^{r+1}e_1\otimes b\otimes b.
\end{align*}
Since $\theta_\ast$ commutes with $d$, it follows from the definitions of $\lambda$ and $Q_1$ (cf.\ \cite{M, CLM}) that $\beta^{(s)}\lambda(u,v)=0$ for $s\le r$ and $\beta^{(r+1)}\lambda(u,v)=Q_1u$. 
\end{proof}

Let $i\colon S^{n-1} \to P^n(p^r)$ denote the inclusion of the bottom cell and $\eta\colon S^n \to S^{n-1}$ the Hopf map.

\begin{lem} \label{homotopy groups}
Let $n\ge 4$ and $r\ge 1$. Then
\begin{enumerate}
\item $\pi_{n-1}(P^n(p^r)) = \mathbb{Z}/p^r \left\langle i \right\rangle$, \smallskip
\item $\pi_n(P^n(p^r)) = \begin{cases} \mathbb{Z}/2 \left\langle i\eta \right\rangle & \text{if } p=2 \\ 0 & \text{if } p \text{ is odd.} \end{cases}$
\end{enumerate}
\end{lem}

\begin{proof}
Both parts follow immediately from the sequence
\[ \pi_j(S^{n-1}) \longrightarrow \pi_j(S^{n-1}) \longrightarrow \pi_j(P^n(p^r)) \longrightarrow \pi_j(S^n) \longrightarrow \pi_j(S^n) \]
induced by the cofibration defining $P^n(p^r)$, which is exact for $j=n-1$, $n$ by the Blakers--Massey theorem. Note that the degree $p^r$ map on $S^{n-1}$ induces multiplication by $p^r$ on $\pi_n(S^{n-1})=\mathbb{Z}/2 \left\langle \eta \right\rangle$ since $\eta$ is a suspension for $n\ge 4$ (whereas $S^2\xrightarrow{2}S^2$ induces multiplication by $4$ on $\pi_3(S^2)$, implying $\pi_3(P^3(2))=\mathbb{Z}/4$, e.g.).
\end{proof}

We are now ready to prove the splittings of Theorem~\ref{splitting thm}, parts (a) and (b) of which are restated below as Lemma~\ref{odd lemma} and Lemma~\ref{even lemma}, respectively.

\begin{lem} \label{odd lemma}
If $p$ is an odd prime and $r \ge 1$, then $D_{p^k}(\Omega^2P^{2n+1}(p^r))$ is stably homotopy equivalent to
\[ P^{2np^k-2}(p^{r+1}) \vee  X_{p^k} \]
for some finite CW-complex $X_{p^k}$ for all $k\ge 1$.
\end{lem}

\begin{proof}
Suppose $p$ is an odd prime, $r\ge 1$ and let $k\ge 1$. By Lemma~\ref{odd higher Bockstein}, the map \[ \delta_k \colon P^{2np^k-2}(p^{r+1}) \longrightarrow \Omega^2P^{2n+1}(p^r) \] induces a monomorphism in mod $p$ homology with
\[ (\delta_k)_\ast(v_{2np^k-2}) = \mathrm{ad}_\lambda^{p^k-1}(v)(u), \quad (\delta_k)_\ast(u_{2np^k-3}) = \beta^{(r+1)}\mathrm{ad}_\lambda^{p^k-1}(v)(u). \]
Since these elements have weight $p^k$ in $H_\ast(\Omega^2P^{2n+1}(p^r))$, by stabilizing $\delta_k$ and composing with the Snaith splitting we obtain a stable map $P^{2np^k-2}(p^{r+1}) \to D_{p^k}(\Omega^2P^{2n+1}(p^r))$ with the same image in homology. 

It therefore suffices to produce a map $f_k \colon D_{p^k}(\Omega^2P^{2n+1}(p^r)) \to P^{2np^k-2}(p^{r+1})$ with 
\begin{equation} \label{f_k}
(f_k)_\ast(\mathrm{ad}_\lambda^{p^k-1}(v)(u)) = v_{2np^k-2}.
\end{equation} 
By collapsing the $(2np^k-4)$-skeleton of $D_{p^k}(\Omega^2P^{2n+1}(p^r))$ to a point, we are left with a complex with cells only in dimensions $2np^k-3$, $2np^k-2$ and $2np^k-1$ (by Lemma~\ref{top homology}) of the form
\begin{equation} \label{cell decomp}
\left(\bigvee_{i=1}^d S^{2np^k-3}\right) \cup e^{2np^k-2} \cup e^{2np^k-2} \cup e^{2np^k-1},
\end{equation}
where $d=\dim H_{2np^k-3}(D_{p^k}(\Omega^2P^{2n+1}(p^r)))$ and the top three cells carry the homology classes $Q_1^kv$, $\beta^{(1)}Q_1^kv$ and $\mathrm{ad}_\lambda^{p^k-1}(v)(u)$. Since $\mathrm{ad}_\lambda^{p^k-1}(v)(u) \in H_{2np^k-2}(D_{p^k}(\Omega^2P^{2n+1}(p^r)))$ supports a nontrivial $(r+1)^\text{st}$ Bockstein by Lemma~\ref{odd higher Bockstein}, we may assume (altering by a self-homotopy equivalence if necessary) that the inclusion of one of the bottom cells in \eqref{cell decomp} has Hurewicz image $\beta^{(r+1)}\mathrm{ad}_\lambda^{p^k-1}(v)(u)$. Then by further collapsing a wedge $\vee_{i=1}^{d-1}S^{2np^k-3}$ of the other bottom cells to a point, we obtain a $4$-cell complex $C$ with mod $p$ homology Bockstein spectral sequence given by
\[
\xymatrix @R=1.2pc @C=1pc{
2np^k-1\qquad & Q_1^kv \ar@/_1pc/[d]_-{\beta^{(1)}} & \\
2np^k-2\qquad & \beta^{(1)}Q_1^kv & \mathrm{ad}_\lambda^{p^k-1}(v)(u) \qquad\quad \ar@/^1pc/[d]^-{\beta^{(r+1)}} \\
2np^k-3\qquad & & \beta^{(r+1)}\mathrm{ad}_\lambda^{p^k-1}(v)(u) \qquad\quad
}
\]
and a map $D_{p^k}(\Omega^2P^{2n+1}(p^r)) \to C$ inducing an epimorphism in homology.

It follows from the description of $H_\ast(C)$ above that $C \simeq P^{2np^k-2}(p^{r+1}) \cup_\alpha e^{2np^k-2} \cup_\gamma e^{2np^k-1}$ for some attaching maps $\alpha$, $\gamma$. Since $\beta^{(1)}Q_1^kv \in H_{2np^k-2}(C)$ is a permanent cycle in the Bockstein spectral sequence and every nonzero element $\alpha \in \pi_{2np^k-3}(P^{2np^k-2}(p^{r+1})) = \mathbb{Z}/p^{r+1}$ is detected by a Bockstein, we conclude that $\alpha$ is trivial. Next we consider 
\[ \gamma \in \pi_{2np^k-2}(P^{2np^k-2}(p^{r+1}) \vee S^{2np^k-2}) = \pi_{2np^k-2}(P^{2np^k-2}(p^{r+1})) \oplus \pi_{2np^k-2}(S^{2np^k-2}). \]
By Lemma \ref{homotopy groups}, $\pi_{2np^k-2}(P^{2np^k-2}(p^{r+1}))=0$ since $p$ is odd, and since the top homology class $Q_1^kv \in H_{2np^k-1}(C)$ supports a nontrivial first Bockstein differential, it follows that the second component of $\gamma$ is of degree $\pm p$. Therefore $C \simeq P^{2np^k-2}(p^{r+1}) \vee P^{2np^k-1}(p)$. Finally, using this splitting we define the map $f_k$ by the composite
\[ f_k\colon D_{p^k}(\Omega^2P^{2n+1}(p^r)) \longrightarrow C \simeq P^{2np^k-2}(p^{r+1}) \vee P^{2np^k-1}(p) \stackrel{\pi_1}{\longrightarrow} P^{2np^k-2}(p^{r+1}), \]
where the first map is the quotient map described in the previous paragraph and $\pi_1$ is the projection onto the first wedge summand. By construction, $f_k$ satisfies \eqref{f_k} so the assertion follows. 
\end{proof}

\begin{lem} \label{even lemma}
If $r>1$, then there is a homotopy equivalence
\[ D_2(\Omega^2P^{2n+1}(2^r)) \simeq P^{4n-2}(2^{r+1}) \vee  X_2 \] 
for some $4$-cell complex $X_2=P^{4n-3}(2^r) \cup CP^{4n-2}(2)$.
\end{lem}

\begin{proof}
Let $r>1$ and note that the mod $2$ homology generators $u$, $v \in H_\ast(\Omega^2P^{2n+1}(2^r))$ in respective degrees $2n-2$, $2n-1$ give a basis for the homology of the first stable summand $D_1(\Omega^2P^{2n+1}(2^r)) = P^{2n-1}(2^r)$ of $\Omega^2P^{2n+1}(2^r)$. Next, a basis for the quadratic part of $H_\ast(\Omega^2P^{2n+1}(2^r))$ is given by the classes
\[
\xymatrix @R=1.2pc @C=2.2pc{
4n-1\qquad & Q_1v \ar@/_1pc/[d]_-{\beta^{(1)}} & \\
4n-2\qquad & v^2 & \lambda(v,u) \quad \ar@/^1pc/[d]^-{\beta^{(r+1)}} \\
4n-3\qquad & uv \ar@/_1pc/[d]_-{\beta^{(r)}} & Q_1u \\
4n-4\qquad & u^2 &
}
\]
with Bockstein differentials acting as indicated by Lemma~\ref{Steenrod squares}, Lemma~\ref{even higher Bockstein} and the fact that $\beta^{(r)}v=u$.
It follows that $D_2(\Omega^2P^{2n+1}(2^r))$ has the homotopy type of a $6$-cell complex
\[ D_2(\Omega^2P^{2n+1}(2^r)) \simeq P^{4n-3}(2^r) \cup_\alpha e^{4n-3} \cup_\gamma e^{4n-2} \cup_\delta e^{4n-2} \cup_\epsilon e^{4n-1} \]
with homology as above, where the bottom Moore space carries the homology classes $uv$ and~$u^2$. As in the proof of Lemma \ref{odd lemma}, the attaching map $\alpha$ is null homotopic since every nonzero element of $\pi_{4n-4}(P^{4n-3}(2^r))=\mathbb{Z}/2^r$ is detected by a Bockstein and $Q_1u$ is a permanent cycle in the homology Bockstein spectral sequence. The next attaching maps $\gamma$ and $\delta$ may therefore be regarded as elements of $\pi_{4n-3}(P^{4n-3}(2^r)\vee S^{4n-3}) = \mathbb{Z}/2 \oplus \mathbb{Z}$, where the first summand is generated by $i\eta$ by Lemma \ref{homotopy groups}. Naturality and the morphism of cofibrations
\[
\xymatrix{
S^{4n-3} \ar[r]^\eta \ar@{=}[d] & S^{4n-4} \ar[r] \ar[d]^-i & \Sigma^{4n-6}\mathbb{C}P^2 \ar[d] \\
S^{4n-3} \ar[r]^-{i\eta} & P^{4n-3}(2^r) \ar[r] & C_{i\eta}
}
\]
imply that $i\eta$ is detected by $Sq^2_\ast$ since $\eta$ is. Since $Sq^2_\ast$ acts trivially on $H_\ast(D_2(\Omega^2P^{2n+1}(2^r)))$ when $r>1$ by Lemma~\ref{Steenrod squares}, the first components of $\gamma$ and $\delta$ must therefore be trivial. Without loss of generality, we may assume the second components of $\gamma$ and $\delta$ are trivial and degree~$\pm 2^{r+1}$, respectively, since we have a basis $\{v^2, \lambda(v,u)\}$ of $H_{4n-2}(D_2(\Omega^2P^{2n+1}(2^r)))$ where $v^2$ is a permanent cycle and $\lambda(v,u)$ supports a nontrivial $\beta^{(r+1)}$. 

We now have a homotopy equivalence
\[ D_2(\Omega^2P^{2n+1}(2^r)) \simeq (P^{4n-3}(2^r) \vee P^{4n-2}(2^{r+1}) \vee S^{4n-2}) \cup _\epsilon e^{4n-1} \]
where $v^2 \in H_{4n-2}(D_2(\Omega^2P^{2n+1}(2^r)))$ corresponds to the fundamental class of the $(4n-2)$-sphere on the right. Denote the components of the attaching map $\epsilon$ by
\[ \epsilon=(\epsilon_1,\epsilon_2,\epsilon_3) \in \pi_{4n-2}(P^{4n-3}(2^r)) \oplus \pi_{4n-2}(P^{4n-2}(2^{r+1})) \oplus \pi_{4n-2}(S^{4n-2}). \]
It suffices to show that $\epsilon_2$ and $\epsilon_3$ are trivial and degree $\pm 2$, respectively. Clearly the first Bockstein $\beta^{(1)}Q_1v =v^2$ on the top class of $D_2(\Omega^2P^{2n+1}(2^r))$ implies $\epsilon_3$ is of degree $\pm 2$. To see that $\epsilon_2$ is trivial, collapse the bottom Moore space $P^{4n-3}(2^r)$ to a point and repeat the argument above analyzing the attaching map $\gamma$.
\end{proof}


To conclude the proof of Theorem~\ref{splitting thm}, it remains to show that $D_2(\Omega^2P^{2n+1}(2))$ is a stably indecomposable $6$-cell complex. This follows immediately from homological considerations: by Lemma~\ref{Steenrod squares} and Lemma~\ref{even higher Bockstein}, $H_\ast(D_2(\Omega^2P^{2n+1}(2)))$ clearly does not admit any nontrivial decomposition respecting Steenrod and higher Bockstein operations. 


\section{Proof of Theorem~\ref{factorization thm}}

In this section we derive Theorem~\ref{factorization thm} from Theorem~\ref{splitting thm} and discuss some implications for the unstable homotopy groups of odd primary Moore spaces.

\begin{proof}[Proof of Theorem \ref{factorization thm}]
Suppose $p$ is an odd prime and $r\ge 1$. Then for each $k\ge 1$, the map $\delta_k\colon P^{2np^k-2}(p^{r+1}) \to \Omega^2P^{2n+1}(p^r)$ from Lemma~\ref{odd higher Bockstein} admits a stable retraction by Lemma~\ref{odd lemma}. Taking adjoints in the resulting homotopy commutative diagram
\[
\xymatrix{
\Sigma^\infty P^{2np^k-2}(p^{r+1}) \ar[r]^{\Sigma^\infty\delta_k} \ar@{=}[rd] & \Sigma^\infty\Omega^2P^{2n+1}(p^r) \ar[d] \\
& \Sigma^\infty P^{2np^k-2}(p^{r+1})
}
\]
yields the desired factorization of the unstable map $E^\infty\colon P^{2np^k-2}(p^{r+1}) \to QP^{2np^k-2}(p^{r+1})$ through $\Omega^2P^{2n+1}(p^r)$.

To similarly factor the stabilization map of a mod $p^{r+1}$ Moore space through $\Omega^2P^{2n}(p^r)$, we reduce to the odd-dimensional case using the fact that $\Omega^2P^{4n-1}(p^r)$ is an unstable retract of $\Omega^2P^{2n}(p^r)$ by the loop space decomposition~\eqref{even decomposition}. Explicitly, there is a map 
\[ P^{(4n-2)p^k-2}(p^{r+1}) \stackrel{\delta_k}{\longrightarrow} \Omega^2P^{4n-1}(p^r) \longrightarrow \Omega^2P^{2n}(p^r) \]
admitting a stable retraction, so the argument above implies that the stabilization map of $P^{(4n-2)p^k-2}(p^{r+1})$ factors through $\Omega^2P^{2n}(p^r)$.
\end{proof}

Note that if $p$ is prime and $\pi_j(P^{2np^k-2}(p^{r+1}))$ is in the stable range so that the map
$E^\infty\colon P^{2np^k-2}(p^{r+1}) \to QP^{2np^k-2}(p^{r+1})$
is an isomorphism on $\pi_j(\;)$, then Theorem~\ref{factorization thm} implies that $\pi_j(P^{2np^k-2}(p^{r+1}))$ is a summand of $\pi_{j+2}(P^{2n+1}(p^r))$. Since for any given $j\in \mathbb{Z}$ we have $\pi_j^s(P^{2np^k-2}(p^{r+1})) = \pi_j(P^{2np^k-2}(p^{r+1}))$ for $k$ sufficiently large, it follows that every stable homotopy group of a mod $p^{r+1}$ Moore space is a summand of $\pi_\ast(P^{2n+1}(p^r))$.

Rephrasing a little, we have the following consequence of Theorem~\ref{factorization thm}. Let $\mathbb{S}/p^{r}$ denote the mod $p^r$ Moore spectrum, that is, the cofibre of $\mathbb{S} \xrightarrow{p^r} \mathbb{S}$ where $\mathbb{S}$ is the sphere spectrum.

\begin{cor} \label{stable unstable}
Let $p$ be an odd prime and $r\ge 1$. Then for each $j\in \mathbb{Z}$, $\pi_j(\mathbb{S}/p^{r+1})$ is a summand of $\pi_{2np^k+j-1}(P^{2n+1}(p^r))$ for every sufficiently large $k$.
\end{cor}

\begin{proof}
For each $j\in \mathbb{Z}$, we have 
\[ \pi_j(\mathbb{S}/p^{r+1}) = \pi_{j+2np^k-3}^s(P^{2np^k-2}(p^{r+1})) = \pi_{j+2np^k-3}(P^{2np^k-2}(p^{r+1})) \] 
for all sufficiently large $k$. Therefore the first commutative diagram in Theorem~\ref{factorization thm} implies that $\pi_j(\mathbb{S}/p^{r+1})$ retracts off $\pi_{j+2np^k-3}(\Omega^2P^{2n+1}(p^r)) = \pi_{j+2np^k-1}(P^{2n+1}(p^r))$.
\end{proof}

\begin{rem}
The second commutative diagram in Theorem~\ref{factorization thm} implies that similar results hold for the unstable homotopy groups of even-dimensional odd primary Moore spaces.
\end{rem}

\section{$v_1$-periodic families} \label{new families section}

In this section we construct new infinite families of higher torsion elements in the unstable homotopy groups of Moore spaces using Theorem~\ref{factorization thm} and periodic self-maps
\[ v_1\colon P^{n+q_r}(p^r) \longrightarrow P^n(p^r), \]
as introduced by Adams \cite{A} in his study of the image of the $J$-homomorphism. Here, 
\[ q_r = \begin{cases} qp^{r-1} & \text{if } p \text{ is odd} \\ \max(8, 2^{r-1}) & \text{if } p=2, \end{cases} \]
where $q=2(p-1)$ and $v_1$ induces an isomorphism in $K$-theory. Such maps exist unstably provided $n\ge 2r+3$ by \cite{DM} and desuspend further to $P^3(p)$ in case $p$ is odd and $r=1$ by \cite{CN}.

Restricting each iterate $v_1^t = v_1 \circ \Sigma^{q_r}v_1 \circ\cdots\circ \Sigma^{tq_r}v_1$ of $v_1$ to the bottom cell gives an infinite family of maps
\[ S^{n+tq_r-1} \longrightarrow P^{n+tq_r}(p^r) \stackrel{v_1^t}{\longrightarrow} P^n(p^r) \]
which generate $\mathbb{Z}/p^r$ summands in $\pi_{n+tq_r-1}(P^n(p^r))$ for $t\ge 0$, and composing with the pinch map $q\colon P^n(p^r) \to S^n$ gives rise to the first studied infinite families in the stable homotopy groups of spheres. For example, if $p$ is odd, these composites form the $\alpha$-family and generate the $p$-component of the image of $J$ in $\pi_{tq_r-1}(\mathbb{S})$ (see \cite[Proposition~1.1]{CK}). To generate $\mathbb{Z}/p^{r+1}$ summands in $\pi_\ast(P^n(p^r))$ when $p$ is odd, we apply the same procedure to mod $p^{r+1}$ Moore spaces and compose into $P^n(p^r)$ along the maps in the diagrams of Theorem~\ref{factorization thm}. 

When $p=2$, analogous unstable maps $\delta_1\colon P^{4n-2}(2^{r+1})\to \Omega^2P^{2n+1}(2^r)$ realizing the stable splittings of Theorem~\ref{splitting thm} only exist when $n=2$ or $4$, as we show below. In the $n=2$ case, no Adams self-map of the mod $2^{r+1}$ Moore spectrum desuspends far enough to precompose $\delta_1$ with. Instead, we show that an infinite family of elements of order $8$ in the homotopy groups of spheres constructed in \cite{KR} factors through $P^6(8)$ and injects along $\delta_1\colon P^6(8) \to \Omega^2P^5(4)$ (see Theorem~\ref{new even families} below). 

\subsection{The odd primary case}

Let $n>1$. As usual, for an odd prime $p$ we let $q=2(p-1)$. The following is a more precise statement of Theorem~\ref{new odd families}.

\begin{thm}
Let $p$ be an odd prime and $r\ge 1$.
\begin{enumerate}
\item If $k \ge \log_p(\frac{r+4}{n})$, then $\pi_{2np^k-1+tqp^r}(P^{2n+1}(p^r))$ contains a $\mathbb{Z}/p^{r+1}$ summand for every $t\ge 0$.
\item If $k \ge \log_p(\frac{r+3}{2n-1})$, then $\pi_{(4n-2)p^k-1+tqp^r}(P^{2n}(p^r))$ contains a $\mathbb{Z}/p^{r+1}$ summand for every $t\ge 0$.
\end{enumerate}

\end{thm}

\begin{proof}
By Theorem~\ref{factorization thm}, $E^\infty\colon P^{2np^k-2}(p^{r+1}) \to QP^{2np^k-2}(p^{r+1})$ factors as a composite
\[ P^{2np^k-2}(p^{r+1}) \stackrel{\delta_k}{\longrightarrow} \Omega^2P^{2n+1}(p^r) \longrightarrow QP^{2np^k-2}(p^{r+1}). \]
Note that the restriction of $\delta_k$ to the bottom cell defines an element of $\pi_{2np^k-1}(P^{2n+1}(p^r))$ of order $p^{r+1}$. The bound on $k$ ensures that $2np^k-2 \ge 2(r+1)+4$, which implies that an unstable representative $v_1\colon P^{2np^k-2+qp^r}(p^{r+1}) \to P^{2np^k-2}(p^{r+1})$ of the Adams map exists by \cite[Proposition~2.11]{DM}. That the restriction of any iterate $v_1^t$ to the bottom cell has order~$p^{r+1}$ follows from the fact that $v_1^t$ induces an isomorphism in $K$-theory. The composite
\[ S^{2np^k-3+tqp^r} \longrightarrow P^{2np^k-2+tqp^r}(p^{r+1}) \stackrel{v_1^t}{\longrightarrow} P^{2np^k-2}(p^{r+1}) \stackrel{\delta_k}{\longrightarrow} \Omega^2P^{2n+1}(p^r) \]
therefore also has order $p^{r+1}$ for all $t\ge 0$ since composing further into $QP^{2np^k-2}(p^{r+1})$ gives the adjoint of the restriction of $\Sigma^\infty v_1^t$ to the bottom cell. Part (b) is proved similarly.
\end{proof}

\begin{rem}
The proof above shows that each $\delta_k$ generates an infinite $v_1$-periodic family in $\pi_\ast(P^{2n+1}(p^r); \mathbb{Z}/p^{r+1})$ giving rise to an infinite family of higher torsion elements in $\pi_\ast(P^{2n+1}(p^r))$. We point out that in the loop space decomposition \cite{CMN3}
\[ \Omega P^{2n+1}(p^r) \simeq T^{2n+1}\{p^r\} \times \Omega \left( \bigvee_\alpha P^{n_\alpha}(p^r) \right), \]
each of these elements lands in the homotopy of the bottom indecomposable factor $T^{2n+1}\{p^r\}$, and many more infinite families than are indicated here can be obtained by applying the Hilton--Milnor theorem to the second factor and iterating our construction above.
\end{rem}

\subsection{The $2$-primary case}

We consider next the problem of desuspending the inclusion of the stable summand $P^{4n-2}(2^{r+1})$ of $\Omega^2P^{2n+1}(2^r)$ given by Theorem~\ref{splitting thm}(b) and mimicking the construction above of unstable $v_1$-periodic families of higher odd primary torsion elements. 

Note that a homotopy commutative diagram
\begin{equation} \label{diagram}
\begin{gathered}
\xymatrix{
P^{4n-2}(2^{r+1}) \ar[r] \ar[rd]_{E^\infty} & \Omega^2P^{2n+1}(2^r) \ar[d] \\
& QP^{4n-2}(2^{r+1})
}
\end{gathered}
\end{equation}
cannot exist unless $r>1$ since $D_2(\Omega^2P^{2n+1}(2))$ is stably indecomposable by Theorem~\ref{splitting thm}(c). Furthermore, such a factorization implies $Q_1u \in H_{4n-3}(\Omega^2P^{2n+1}(2^r))$ is spherical since only this class lies in the image of the $(r+1)^\text{st}$ Bockstein in degree $4n-3$. For $r=1$, it follows from the proposition below that this class is spherical only in Kervaire invariant dimensions. 

\begin{prop}[{\cite[Proposition~2.21]{W}}] \label{Wu's prop}
The class $u^2\in H_{4n-2}(\Omega P^{2n+1}(2))$ is spherical if and only if the Whitehead square $w_{2n-1}\in \pi_{4n-3}(S^{2n-1})$ is divisible by $2$.
\end{prop}

For $r>1$, the same argument leads to the following.

\begin{prop} \label{spherical classes}
Let $r>1$. The following conditions are equivalent:
\begin{enumerate}
\item $Q_1u \in H_{4n-3}(\Omega^2P^{2n+1}(2^r))$ is spherical; 
\item $u^2 \in H_{4n-2}(\Omega P^{2n+1}(2^r))$ is spherical;
\item $n=1$, $2$ or $4$.
\end{enumerate}
\end{prop}

\begin{proof}
If a map $f\colon S^{4n-3} \to \Omega^2P^{2n+1}(2^r)$ has mod $2$ reduced Hurewicz image $Q_1u$, then the adjoint of $f$ factors as
\[ f'\colon S^{4n-2} \stackrel{\Sigma f}{\longrightarrow} \Sigma\Omega^2P^{2n+1}(2^r) \stackrel{\sigma}{\longrightarrow} \Omega P^{2n+1}(2^r), \]
where $\sigma$ induces the homology suspension $\sigma_\ast \colon H_\ast(\Omega^2P^{2n+1}(2^r)) \to H_{\ast+1}(\Omega P^{2n+1}(2^r))$. Thus $\sigma_\ast(Q_1u)=u^2$ is the Hurewicz image of $f'$. 

Conversely, given $g'\colon S^{4n-2} \to \Omega P^{2n+1}(2^r)$ with $g'_\ast(\iota_{4n-2})=u^2$, the adjoint of $g'$ factors as
\[ g\colon S^{4n-3} \stackrel{E}{\longrightarrow} \Omega S^{4n-2} \stackrel{\Omega g'}{\longrightarrow} \Omega^2P^{2n+1}(2^r). \]
Consider the morphism of path-loop fibrations induced by $g'$. Since $u^2$ transgresses to $Q_1u$ in the Serre spectral sequence associated to the path-loop fibration over $\Omega P^{2n+1}(2^r)$, it follows by naturality that $g_\ast(\iota_{4n-3})=Q_1u$. Therefore conditions (a) and (b) are equivalent.

If $n=1$, $2$ or $4$, then the adjoint of the Hopf invariant one map $S^{4n-1} \to S^{2n}$ has Hurewicz image $\iota_{2n-1}^2 \in H_{4n-2}(\Omega S^{2n})$, so the composite $S^{4n-2} \to \Omega S^{2n} \xrightarrow{\Omega i\,} \Omega P^{2n+1}(2^r)$ has Hurewicz image $u^2$.

Conversely, if $u^2 \in H_{4n-2}(\Omega P^{2n+1}(2^r))$ is spherical, then the proof given in \cite{W} of Proposition~\ref{Wu's prop} above shows that $i\circ w_{2n-1}$ is null homotopic in the diagram
\[
\xymatrix{
\Omega P^{2n}(2^r) \ar[r] & E \ar[r]^-f & S^{2n-1} \ar[r]^-i & P^{2n}(2^r) \\
 & & S^{4n-3} \ar[u]_{w_{2n-1}} \ar@{-->}[ul]^-\ell
}
\]
where the top row is a fibration sequence. By \cite[Lemma~21.1]{C course}, $[u,v] \in H_{4n-3}(\Omega P^{2n}(2^r))$ is spherical and so is its image in $H_{4n-3}(E)$. As in \cite{W}, it follows that the $(4n-3)$-skeleton of~$E$ is homotopy equivalent to $S^{2n-1}\vee S^{4n-3}$ and $f|_{S^{4n-3}}$ is null homotopic. Therefore a lift~$\ell$ may be chosen to factor through $f|_{S^{2n-1}}$, which is of degree $2^r$. Since the degree $2$ map induces multiplication by $2$ on $\pi_{4n-3}(S^{2n-1})$ by Barratt's distributivity formula \cite[Proposition~4.3]{C course}, the Whitehead square $w_{2n-1}$ is divisible by $2^r$, which implies $n\in \{1,2,4\}$ since $r>1$. 
\end{proof}

By Proposition~\ref{spherical classes}, diagrams of the form \eqref{diagram} inducing the stable splittings of Theorem~\ref{splitting thm}(b) (where $n>1$ is assumed) cannot exist if $n\neq 2$ or $4$. We verify that such diagrams do exist in these two exceptional dimensions.

\begin{thm} \label{even factorization thm}
Let $r>1$. Then there exist homotopy commutative diagrams
\[
\xymatrix{
P^6(2^{r+1}) \ar[r] \ar[rd]_{E^\infty} & \Omega^2P^5(2^r) \ar[d] \\
& QP^6(2^{r+1})
} \qquad\qquad
\xymatrix{ 
P^{14}(2^{r+1}) \ar[r] \ar[rd]_{E^\infty} & \Omega^2P^9(2^r) \ar[d] \\
& QP^{14}(2^{r+1}).
}
\]
\end{thm}

\begin{proof}
Let $n=2$ or $4$ and let $f\colon S^{4n-3} \to \Omega^2P^{2n+1}(2^r)$ be a map with mod $2$ reduced Hurewicz image $Q_1u$. Then by Lemma~\ref{even higher Bockstein}, the integral Hurewicz image of $f$ is a generator of $H_{4n-3}(\Omega^2P^{2n+1}(2^r); \mathbb{Z})\cong \mathbb{Z}/2^{r+1}$, so $f$ has order at least $2^{r+1}$. That $f$ has order at most~$2^{r+1}$ follows from \cite[Proposition~13.3]{C course}, so $f$ extends to a map $\bar{f}\colon P^{4n-2}(2^{r+1}) \to \Omega^2P^{2n+1}(2^r)$ with $\beta^{(r+1)}\bar{f}_\ast(v_{4n-2})=Q_1u$. Since $r>1$, the Snaith splitting and Theorem~\ref{splitting thm}(b) give a composite
\[ \Omega^2P^{2n+1}(2^r) \longrightarrow QD_2(\Omega^2P^{2n+1}(2^r)) \longrightarrow QP^{4n-2}(2^{r+1}) \]
which is an epimorphism on $H_{4n-2}(\:)$ and $H_{4n-3}(\:)$. It follows that the composition of $\bar{f}$ with the composite above is $(4n-2)$-connected and hence homotopic to the stabilization map $E^\infty\colon P^{4n-2}(2^{r+1}) \to QP^{4n-2}(2^{r+1})$. 
\end{proof}

\begin{thm} \label{new even families}
Let $r=2$ or $3$. Then
\begin{enumerate}
\item $\pi_{3+8t}(P^5(4))$ contains a $\mathbb{Z}/8$ summand for every $t\ge 1$;
\item $\pi_{7+8t}(P^9(2^r))$ contains a $\mathbb{Z}/2^{r+1}$ summand for every $t\ge 1$.
\end{enumerate}
\end{thm}

\begin{proof}
We begin with part (b). Consider the second diagram in Theorem~\ref{even factorization thm} and let $r=2$ or~$3$. Then for stability reasons, an unstable Adams map $v_1\colon P^{n+8}(2^{r+1}) \to P^n(2^{r+1})$ exists for $n=14$. As in the odd primary case, restricting any iterate $v_1^{t-1}\colon P^{6+8t}(2^{r+1}) \to P^{14}(2^{r+1})$ to the bottom cell yields a homotopy class of order $2^{r+1}$ and stable order $2^{r+1}$. The resulting composition with the map $P^{14}(2^{r+1}) \to \Omega^2P^9(2^r)$ therefore generates a $\mathbb{Z}/2^{r+1}$ summand in $\pi_{5+8t}(\Omega^2P^9(2^r))=\pi_{7+8t}(P^9(2^r))$ by Theorem~\ref{even factorization thm}.

For part (a), we use the fact that an unstable Adams map $v_1\colon P^{n+8}(8) \to P^n(8)$ exists for $n\ge 9$ with the property that the composite
\[ N_t\colon S^{8t} \stackrel{i}{\longrightarrow} P^{1+8t}(8) \stackrel{v_1^{t-1}}{\longrightarrow} P^9(8) \stackrel{\nu^\sharp}{\longrightarrow} S^5 \]
has order $8$ in $\pi_{8t}(S^5)$ for all $t\ge 1$ by \cite[Theorem~E]{KR}. Here $\nu^\sharp$ denotes an extension of $\nu\colon S^8\to S^5$. Suspending once, an extension $\nu^\sharp\colon P^{10}(8) \to S^6$ of $\nu\colon S^9 \to S^6$ can be chosen to factor through the pinch map $q\colon P^6(8)\to S^6$ (we postpone a proof of this claim to Lemma~\ref{nu factorization} below). Combining this with the first diagram in Theorem~\ref{even factorization thm}, we obtain a homotopy commutative diagram
\[
\xymatrix{
P^{2+8t}(8) \ar[r]^{v_1^{t-1}} & P^{10}(8) \ar[r] & P^6(8) \ar[r] \ar[rd]_{E^\infty} \ar[d]^q & \Omega^2P^5(4) \ar[d] \\
S^{1+8t} \ar[u]_i \ar[rr]^{\Sigma N_t} & & S^6 \ar[rd]_{E^\infty} & QP^6(8) \ar[d]^{Qq} \\
 & & & QS^6.
}
\]
The composite $E^\infty \circ \Sigma N_t$ is adjoint to $\Sigma^\infty N_t \in \pi_{8t-5}(\mathbb{S})$ and therefore has order $8$ since the proof of \cite[Theorem~E]{KR} shows $N_t$ has real $e$-invariant $\frac{b}{8}$ where $b$ is odd. Hence the composite $S^{1+8t} \to \Omega^2P^5(4)$ has order at least $8$. Since $8i=0$, the theorem follows.
\end{proof}

It remains to prove the following factorization of $\nu\in \pi_9(S^6)$ used in the proof above.

\begin{lem} \label{nu factorization}
There is a homotopy commutative diagram
\[
\xymatrix{
P^{10}(8) \ar[r] & P^6(8) \ar[d]^q \\
S^9 \ar[u]_i \ar[r]^\nu & S^6.
}
\]
\end{lem}

\begin{proof}
Since $8\circ \nu=\nu\circ 8=0$ in $\pi_9(S^6)$, $\nu$ lifts to the fibre $S^6\{8\}$ of the degree $8$ map, and since $\mathrm{sk}_9(S^6\{8\})\simeq P^6(8)$, it follows that $\nu$ factors as $S^9 \stackrel{\ell\,}{\to} P^6(8) \stackrel{q\,}{\to} S^6$. It suffices to show that $\ell$ has order $8$. The fibre $F$ of the pinch map $q$ has the homotopy type of a CW-complex $S^5 \cup e^{10} \cup e^{15}\cup\cdots$ where the first attaching map is $8w_5=0$ by \cite[Corollary~5.8]{G}. It follows that $\pi_9(F)=\mathbb{Z}/2$. The short exact sequence
\[ 0=\pi_{10}(S^6) \longrightarrow \pi_9(F) \longrightarrow \pi_9(P^6(8)) \stackrel{q_\ast}{\longrightarrow} \pi_9(S^6) \stackrel{0}{\longrightarrow} \pi_8(F) \]
therefore implies $\pi_9(P^6(8))=\mathbb{Z}/2 \oplus \mathbb{Z}/8$. In particular, $\ell$ has order $8$.
\end{proof}

\end{document}